\documentclass[letterpaper,12pt]{amsart}
\usepackage{amsmath,amssymb,amsxtra,amsthm, amstext, amscd,amsfonts,fancyhdr, hyperref,enumerate}
\hypersetup{
colorlinks=true,
urlcolor=black,
citecolor=blue,
linkcolor=blue,
}

\newcommand{\bA}{{\mathbb{A}}}

\newcommand{\bC}{{\mathbb{C}}}

\newcommand{\bF}{{\mathbb{F}}}

\newcommand{\bN}{{\mathbb{N}}}

\newcommand{\bQ}{{\mathbb{Q}}}
\newcommand{\bR}{{\mathbb{R}}}

\newcommand{\bZ}{{\mathbb{Z}}}

\newcommand{\Be}{{\mathbf{e}}}
\newcommand{\Bf}{{\mathbf{f}}}

\newcommand{\Bm}{{\mathbf{m}}}

\newcommand{\Bs}{{\mathbf{s}}}

\newcommand{\Bx}{{\mathbf{x}}}

\newcommand{\BB}{\mathbf{B}}

  \newcommand{\E}{{\mathcal{E}}}

  \newcommand{\M}{{\mathcal{M}}}
  \newcommand{\N}{{\mathcal{N}}}

\renewcommand{\P}{{\mathcal{P}}}
  \newcommand{\Q}{{\mathcal{Q}}}

  \newcommand{\UUU}{{\mathcal{U}}}

  \newcommand{\Y}{{\mathcal{Y}}}
  \newcommand{\Z}{{\mathcal{Z}}}

\newcommand{\UU}{\mathbf{U}}

\newcommand{\ep}{\varepsilon}

\newcommand{\upchi}{{\raise.35ex\hbox{$\chi$}}}

\makeatletter
\@namedef{subjclassname@2010}{%
  \textup{2010} Mathematics Subject Classification}
\makeatother

\newtheorem{theorem}{Theorem}[section]

\newtheorem{lemma}[theorem]{Lemma}

\theoremstyle{definition}

\theoremstyle{remark}

\numberwithin{equation}{section}

\begin{document}

\title{Density of power-free values of polynomials}

\author{Kostadinka Lapkova}
\address{Institute of Analysis and Number Theory\\
Graz University of Technology \\
Kopernikusgasse 24/II\\
8010 Graz\\
Austria}
\email{lapkova@math.tugraz.at}

\author{Stanley Yao Xiao}
\address{Department of Mathematics \\
University of Toronto \\
Bahen Centre \\
40 St. George St., Room 6290 \\
Toronto, Ontario, Canada M5S 2E4}
\email{stanley.xiao@math.toronto.edu}
\indent


\begin{abstract}We establish asymptotic formulae for the number of $k$-free values of square-free polynomials $F(x_1,\cdots,x_n)\in\bZ[x_1,\cdots,x_n]$ of degree $d\geq 2$ for any $n\geq 1$, including when the variables are prime, as long as $k\geq (3d+1)/4$. This generalizes a work of Browning. 
\end{abstract}

\maketitle

\tableofcontents

\section{Introduction}
\label{Intro}

In this paper we consider power-free values of polynomials $F(x_1, \cdots, x_n)$ with integer coefficients and degree $d \geq 2$. Put
\begin{equation} \label{main} N_{F,k}(B) = \# \{(x_1, \cdots, x_n) \in \bZ^n : |x_i| \leq B \text{ for } i = 1, \cdots, n, F(x_1, \cdots, x_n) \text{ is } k \text{-free}\}.
\end{equation}
Our goal is to show that $N_{F,k}(B)$ satisfies an asymptotic formula provided that $F$ is not always divisible by $p^{k}$ for a fixed prime $p$, and $k$ is suitably large compared to $d$. \\ \\
The case $k = d-1$ is of particular interest. In the case of polynomials in a single variable, the first to establish the infinitude of $N_{F,d-1} (B)$ was Erd\H{o}s \cite{Erd}. His argument did not establish an asymptotic formula for $N_{F,d-1}(B)$; this had to wait until Hooley \cite{Hoo1}. In the two variable case, the asymptotic formula for $N_{F,d-1}(B)$ is only established by Hooley \cite{Hoo3} in the case when $F$ splits into linear factors over some finite extension of $\bQ$. In this paper we provide an asymtotic formula for $N_{F,d-1} (B)$ for any number of variables $n\geq 1$ and any square-free polynomial $F$ with degree $d\geq 5$.\\ \\
For general $k$ and $n = 1$, various authors have worked on the problem of estimating $N_{F,k}(B)$. The record is a theorem of Browning \cite{B2}, which asserts that the expected asymptotic formula for $N_{F,k}(B)$ holds when $k \geq (3d+1)/4$. 
The main point of our paper is to reduce the problem for general $n$ to the setting of Browning's theorem. \\ 

In the case of multiple variables, most of the work has been done in the case of binary forms only. The asymptotic formula for $N_{F,k}(B)$ for binary forms $F$ was established for $k \geq (d-1)/2$ by Greaves \cite{Gre}, $k > (2\sqrt{2} - 1)d/4$ by Filaseta \cite{Fil}, $k > 7d/16$ by Browning \cite{B2}, and $k > 7d/18$ by Xiao \cite{X1}. For polynomials of more variables, Bhargava handled square-free values for discriminants of representations of some prehomogeneous vector spaces in \cite{Bha}, and Bhargava, Shankar, Wang handled the case of discriminants of polynomials in \cite{BSW}. Xiao handled the case of square-free values of decomposable forms in \cite{X2}. In fact, he obtained an asymptotic relation for $N_{F,2}(B)$ when $F$ is decomposable whenever $d \leq 2n + 2$. \\ 

For inhomogeneous polynomials of two variables the works of Hooley \cite{Hoo2} and Browning \cite{B2} provided \emph{lower bounds} for the number of $k$-free values and Hooley \cite{Hoo3} managed to provide an asymptotic formula in certain cases. There are also several specific inhomogeneous polynomials of more variables whose power-free values were estimated asymptotically, e.g.  \cite{leboudec}, \cite{lapkova}. We should mention that already Poonen \cite{Po} shows that the number of square-free values of multivariable polynomials $F$ has a positive density, assuming the \emph{abc}-conjecture, however the considered density is differently defined than the one arising when we evaluate asymptotically \eqref{main}.\\

The main result in this paper is the following theorem, which asserts that an asymptotic relation for $N_{F,k}(B)$ holds whenever $k \geq (3d+1)/4$, for any $n\geq 1$ and assuming only a necessary condition on the divisibility of the square-free polynomial $F$. Observe that the lower bound for $k$ is the same as in Browning's theorem in \cite{B2}. 

\begin{theorem} \label{MT1} Let $k \geq 2$ be a positive integer and let $F$ be a square-free polynomial with integer coefficients and degree $d \geq 2$ in $n$ variables, such that for all primes $p$, there exists an integer $n$-tuple $(m_1, \cdots, m_n)$ such that $p^k \nmid F(m_1, \cdots, m_n)$. Then there exists a positive number $C_{F,k}$ such that the asymptotic relation
\[N_{F,k}(B) \sim C_{F,k} B^n\]
holds whenever $k \geq (3d+1)/4$. 
\end{theorem}
Here the constant term is given by the limit of an absolutely convergent infinite product
\[C_{F,k}=\prod_p \left(1-\frac{\rho_F(p^k)}{p^{kn}}\right),\]
where
\[\rho_F(m) = \#\{(s_1, \cdots, s_n) \in (\bZ/m\bZ)^n : F(s_1, \cdots, s_n) \equiv 0 \pmod{m} \}.\]
Under our assumptions $C_{F,k}$ is positive. (The convergence is a well-known fact when $n=1$ and follows from Lemma \ref{rhoF est} for $n\geq 2$.) In particular, whenever $d \geq 5$, the quantity $N_{F,d-1}(B)$ will satisfy the expected asymptotic formula. \\ 

Following Browning \cite{B2}, we can also handle the case when we restrict the inputs to be primes. This extends an Erd\H{o}s conjecture for $(d-1)$-free values of one-variable polynomials at prime arguments to multi-variable polynomials with $d\geq 5$. We thus obtain the following theorem:

\begin{theorem} \label{MT2} Let $k \geq 2$ be a positive integer and let $F$ be a square-free polynomial with integer coefficients and degree $d \geq 2$ in $n$ variables, such that for all primes $p$, there exists an integer $n$-tuple $(m_1, \cdots, m_n)$ such that $p^k \nmid F(m_1, \cdots, m_n)$. Put
\[\N_{F,k}(B) = \# \{(p_1, \cdots, p_n) \in \bZ^n : |p_i| \leq B \text{ for } 1 \leq i \leq n, F(p_1, \cdots, p_n) \text{ is }k\text{-free}\} ,\]
where $p_i$ is prime for $1 \leq i \leq n$.
Then there exists a positive number $C_{F,k}'$ such that the asymptotic relation
\[\N_{F,k}(B) \sim C_{F,k}' \frac{B^n}{(\log B)^n}\]
holds whenever $k \geq (3d+1)/4$. 
\end{theorem}
Here we have 
\[C'_{F,k}=\prod_p \left(1-\frac{\rho^*_F(p^k)}{\phi(p^k)^n}\right),\]
where 
\begin{eqnarray}\label{prime rho} \rho_F^\ast(m) &=& \# \{\Bs \in \bZ^n : |s_i| \leq m-1,\text{ for } 1 \leq i \leq n, \\
&&\gcd(s_i,s_j) = 1 \text{ for all } i < j , F(\Bs) \equiv 0 \pmod{m}\}\nonumber\end{eqnarray}
and $C_{F,k}'$ is again positive under our assumptions.\\

The proof of Theorems \ref{MT1} and \ref{MT2} will largely rely on the affine determinant method of Heath-Brown \cite{HB2}, which is the same tool used by Browning in \cite{B2}. The main innovation in this paper is using sieve methods to handle medium sized primes which contribute to $N_{F,k}(B)$ and $\N_{F,k}(B)$. 
%
%
%

\section{Preliminaries: the simple and Ekedahl's sieves}

\subsection{Counting $k$-free values over integer inputs}

We shall find quantities $N_1(B)$, $N_2(B)$, $N_3(B)$ such that
\begin{equation} \label{simple sieve} N_1(B) - N_2(B) - N_3(B) \leq N_{F,k}(B) \leq N_1(B),
\end{equation}
and for which $N_2(B), N_3(B) = o(B^n)$. This is similar to the \emph{simple sieve} technique of Hooley. \\

Let $\xi_1 = \frac{1}{nk} \log B$ and $\xi_2 = B (\log B)^{1/2}$, and put
\begin{equation} \label{N1} N_1(B) =N_{1,k}(B)= \# \{(x_1, \cdots, x_n) \in \bZ^n : |x_i| \leq B, p^{k} | F(x_1, \cdots, x_n) \Rightarrow p > \xi_1 \}, \end{equation}

\begin{equation} \label{N2} N_2(B) =N_{2,k}(B)= \# \{(x_1, \cdots, x_n) \in \bZ^n : |x_i| \leq B, p^{k} | F(x_1, \cdots, x_n) \Rightarrow p > \xi_1, 
\end{equation} 
\[\exists \xi_1 < p \leq \xi_2 \text{ s.t. } p^2 | F(x_1, \cdots, x_n)\}\]
and
\begin{equation} \label{N3} N_3(B) = N_{3,k}(B)=\# \{(x_1, \cdots, x_n) \in \bZ^n : |x_i| \leq B, p^{k} | F(x_1, \cdots, x_n) \Rightarrow p > \xi_1, \end{equation}
\[p^2 \nmid F(x_1, \cdots, x_n) \text{ for } \xi_1 < p \leq \xi_2, \exists p > \xi_2 \text{ s.t. } p^{k} | F(x_1, \cdots, x_n)\}.\]
Then it is clear that (\ref{simple sieve}) holds. \\

We first show that $N_1(X)$ gives us a term with the expected order of magnitude. We require an estimate for $\rho_F(p^k)$ given by the following lemma:
\begin{lemma} \label{rhoF est} Let $F$ be a square-free polynomial in $n$ variables with integer coefficients, and such that for all primes $p$, $p^k$ does not divide $F$ identically. Then for any $k\geq 2$ we have $\rho_F(p^k) = O_F\left(p^{nk-2}\right)$. 
\end{lemma}

\begin{proof} For a positive integer $m$, let $S_m(F)$ denote the set of points (over $\bC$, say) of $F$ which have multiplicity $m$. It is clear that $m \leq \deg F$. By our hypothesis on $F$ it follows that for $m \geq 2$ $S_m(F)$ is not Zariski dense in the variety $X_F: F = 0$, hence as a subvariety of $X_F$ it has codimension at least one. For all but finitely many primes $p$, we have that any smooth point $\Bx \in X_F \cap \bZ^n$ over $\bC$ reduces to a smooth point over $\bF_p$; moreover, it follows from standard arguments that smooth points over $\bF_p$ contribute at most $C p^{n(k-1)}$ to $\rho_F(p^k)$. \\

We now estimate the contribution from $S_m(F)$ for $m \geq 2$. For any such point $\Bx_0 \in X_F \cap \bZ^n$, we can take a Taylor expansion around $\Bx_0$ and see that $p^k | F(\Bx)$ whenever $\Bx \equiv \Bx_0 \pmod{p^{\lceil k/m \rceil}}$. Thus, working $(\bZ/p^k \bZ)^n$ and taking $e = \lceil k/m \rceil$, we see that each $\Bx_0 \pmod{p^e}$ which is the reduction of a point in $S_m(F)$ mod $p^e$ contributes $O(p^{(k-e)n}) = O(p^{kn - en})$ to $\rho_F(p^k)$. \\

Now, the number of such points mod $p^e$ can be counted as follows. First begin with the set $S_m(F)(p)$, the set of points of $X_F \pmod{p}$ which has multiplicity $m$. There are $O(p^{n-2})$ such points, since $S_m(F)(p)$ has codimension one in $X_F$ by hypothesis. For each such point there are at most $O(p^{(e-1)n})$ points in $(\bZ/p^e \bZ)^n$ which lies above it. Thus there are $O(p^{(e-1)n + n - 2})$ such points. Thus there are
\[O\left(p^{kn - en + (e-1)n + n-2}\right) = O\left(p^{kn -2}\right)\]
such points, as desired.  
\end{proof}

Now put 
\[N(B ; h) = \# \{\Bx \in \bZ^n: |x_i| \leq B, h^k | F(x_1, \cdots, x_n)\}.\]
By standard properties of the M\"obius function $\mu$ we then see that
\begin{align} \label{N1 mob} N_1(B) & = \sum_{\substack{ h \in \bN \\ p | h \Rightarrow p \leq \xi_1}} \mu(h) N(B; h) \\
& = \sum_{\substack{h \in \bN \\ p | h \Rightarrow p \leq \xi_1}} \mu(h) \rho_F(h^k) \left\{ \frac{B^n}{h^{nk}} + O \left(\frac{B^{n-1}}{h^{k(n-1)}} + 1 \right) \right \} \notag.
\end{align}
Since the summand vanishes when $h$ is not square-free, we may assume that $h$ is in fact square-free. Hence
\[h \leq \prod_{p \leq \xi_1} p = \exp \left(\sum_{p \leq \xi_1} \log p \right) \leq e^{2 \xi_1}.\]
By Lemma \ref{rhoF est}, we then see that
\begin{equation*}  N_1(B) = B^n \prod_{p \leq \xi_1} \left(1 - \frac{\rho_F(p^k)}{p^{nk}}\right) + O \left(\sum_{h \leq e^{2 \xi_1}} h^{kn - 2 + \ep} \left(B^{n-1} h^{-k(n-1)} + 1 \right) \right). \end{equation*}
The big-$O$ term can be evaluated to be
\[O_\ep \left(B^\ep \left(B^{n-1 + \frac{k-1}{nk}} + B^{\frac{kn-1}{nk} } \right)\right),\]
and we see that this is $O(B^{n - \delta})$ for some $\delta > 0$. Thus we have
\begin{equation}\label{N1 est} N_1(B)=C_{F,k}B^n + O \left(B^n \xi_1^{-1} \right) +O(B^{n-\delta}).\end{equation}

Next we give an estimate for $N_2(B)$. 
\begin{lemma} \label{N2 est} Let $N_2(B)$ be as in (\ref{N2}). Then
\[N_2(B) = O_d \left(\frac{B^n}{\xi_1} + \frac{B^{n-1} \xi_2}{\log \xi_2}\right).\]
\end{lemma}

To prove this lemma, we will need the following result due to Ekedahl; the formulation below is due to Bhargava and Shankar \cite{BhaSha}.  

\begin{lemma}[Ekedahl] \label{Ekedahl} Let $\mathcal{B}$ be a compact region in $\bR^n$ with positive measure, and let $Y$ be any closed subscheme of $\bA_\bZ^n$ of co-dimension $k \geq 2$. Let $r$ and $M$ be positive real numbers. Then we have
\[\#\{v \in r\mathcal{B} \cap \bZ^n : v \pmod{p} \in Y(\bF_p) \text{ for some } p > M\} = O \left(\frac{r^n}{M^{k-1} \log M} + r^{n-k+1} \right).\]
\end{lemma}

\begin{proof}[Proof of Lemma \ref{N2 est}] We fix the variables $\Bx = (x_2, \cdots, x_n)$, and for each such choice, we put
\[F_\Bx(x) = F(x, x_2, \cdots, x_n).\]
Define the function, for a polynomial $f$ in a single variable $x$, by
\[\rho_f(m) = \#\{s \in \bZ/m\bZ : f(s) \equiv 0 \pmod{m}\}.\]
It is clear from the Chinese Remainder Theorem that $\rho_f$ is multiplicative. Observe that $\rho_f(p) \leq d$ if $p$ does not divide all coefficients of $f$ and in this case the bound is independent of $p$. Moreover, if $p \nmid \Delta(f)$, then for any positive integer $k$ we have $\rho_f(p^k) \leq d$. We now use Lemma \ref{Ekedahl} as follows. Let $G(x_1, \cdots, x_n) = \frac{\partial F}{\partial x_1} (x_1, \cdots, x_n)$. Define the variety $V_{F,G}$ to be 
\begin{equation} \label{VFG} V_{F,G} = \{(x_1, \cdots, x_n) \in \bC^n : F(x_1, \cdots, x_n) = G(x_1, \cdots, x_n) = 0\}.
\end{equation}
Observe that $V_{F,G}$ is of co-dimension two and is defined over $\bZ$. Put 
\[\N_p^\ast(B) = \{(x_1, \cdots, x_n) \in \bZ^n : |x_i| \leq B, (x_1, \cdots, x_n) \pmod{p} \in V_{F,G}(\bF_p)\}.\]
It follows from Lemma \ref{Ekedahl} that
\[\# \bigcup_{p > \xi_1} \N_p^\ast(B) = O \left(\frac{B^n}{\xi_1 \log \xi_1} + B^{n-1} \right),\]
which is an acceptable error term. \\ \\
Now put
\[N_p^\dagger(B) = \#\{(x_1, \cdots, x_n) \in \bZ^n : |x_i| \leq B, p^2 | F(x_1, \cdots, x_n), p \nmid G(x_1, \cdots, x_n)\}.\]
It then follows that
\[N_2(B) \leq \sum_{\xi_1 < p \leq \xi_2} N_p^\dagger(B) + \sum_{\xi_1 < p \leq \xi_2} N_p^\ast(B),\]
and since we have already estimated the second sum, it suffices to estimate the former. For fixed $(x_2, \cdots, x_n)$, the solutions to $f(x) = F(x, x_2, \cdots, x_n) \equiv 0 \pmod{p^2}$ contributing to $N_p^\dagger(X)$ must satisfy $p \nmid \Delta(f)$; in particular, the number of solutions in $\bZ/p^2\bZ$ is at most $d$.  We then have that
\[\sum_{\xi_1 < p \leq \xi_2} N_p^\dagger(B) = O_d \left(B^{n-1} \sum_{\xi_1 < p \leq \xi_2} \left(\frac{B}{p^2} + 1 \right) \right) = O_d \left(\frac{B^n}{\xi_1} + \frac{B^{n-1} \xi_2}{\log \xi_2}\right).\]
\end{proof}

\subsection{Counting $k$-free values over prime inputs} We shall seek an analogue to (\ref{simple sieve}) for the case of prime inputs. Put $\N_1(B)$, $\N_2(B)$, $\N_3(B)$ for the analogues to (\ref{N1}), (\ref{N2}), and (\ref{N3}) in the case of prime inputs. We begin with an estimate for $\N_1(B)$. Since we are interested in prime inputs, we should modify the function $\rho_F(m)$. In particular, the only way for $\Bm$ with all prime coordinates to have two coordinates to share a common factor is if they are equal, and thus $\Bm$ lies on one of the hyperplanes in $\bR^n$ defined by $x_i = x_j$ for some $i < j$. These points are negligible in the box $[-B,B]^n \cap \bZ^n$, so we may assume that $\gcd(m_i, m_j) = 1$ for all $i < j$ and consider the modified quantity \eqref{prime rho}.

It is immediate that $\rho_F^\ast(p^k) = O_{d,n} \left(p^{kn-2)} \right)$. Moreover, the number of possible elements in $(\bZ/p^k \bZ)^n$ such that no coordinate is divisible by $p$ is $\phi(p^k)^n = (p^k - p^{k-1})^n$. Similar to (\ref{N1 mob}), we find that
\[\N_1(B) = \sum_{\substack{p \in \bN \\ p | h \Rightarrow p > \xi_1'}} \mu(h) \rho_F^\ast(p^k) \left(\frac{B^n}{\phi(p^k)^n (\log B)^n} + O_c\left(\frac{B^{n}}{\phi(p^k)^{n-1} (\log B)^{n-1} \exp(c \sqrt{\log B})}  \right) \right), \]
by the Siegel-Walfisz theorem. We thus find that
\[\N_1(B) = \frac{B^n}{(\log B)^n} \prod_{p \leq \xi_1'} \left(1 - \frac{\rho_F^\ast(p^k)}{\phi(p^k)^n} \right) + O_c \left(\frac{B^n (\xi_1')^{nk}}{(\log B)^{n-1} \exp(c\sqrt{\log B})} \right). \]
Since $\exp(c \sqrt{\log B})$ is eventually larger than any power of $\log B$, we may take $\xi_1' = (\log B)^{n^2 k^2}$ say. We then see that $\N_1(B)$ approximates our main term, bearing in mind that in a similar way we have $C'_{F,k}>0$. Now the proof of the analogue of Lemma \ref{N2 est} for $\N_2(B)$ carries through as before, with $\xi_1$ replaced with $\xi_1'$ and $\xi_2$ replaced with $B \exp(-c \sqrt{\log B})$.  The remaining quantity $\N_3(B)$ will be estimated by the upper bound of $N_3(B)$, which is obtained at the end of the paper in subsection \ref{subsect N3}. \\ \\
In the next section, we shall derive a version of the global determinant method which applies to affine varieties, refining Heath-Brown's Theorem 15 in \cite{HB2}. It is well-known that the global determinant method produces estimates which are uniform in the coefficients of the polynomials involved and produces power-saving error terms. Thus, the proof of Theorems \ref{MT1} and \ref{MT2} in the case $d > 4$ can be carried out simultaneously. 

\section{Global affine determinant method}

In \cite{HB1}, Heath-Brown introduced a novel technique which is widely applicable to the subject of counting rational points on algebraic varieties, which is now known as the $p$-adic determinant method. It is a generalization of the original determinant method of Bombieri and Pila, which is based on real analytic arguments. This was later refined by Salberger in \cite{S2}. Salberger named his refinement the \emph{global} determinant method, in reference to the fact that it uses multiple primes simultaneously. \\ \\
In \cite{HB2}, Heath-Brown gave a different refinement of his $p$-adic determinant method in \cite{HB1}. Originally the $p$-adic determinant method could only treat projective hypersurfaces, but in \cite{HB2} Heath-Brown showed that an analogous version exists for affine varieties. Browning combined Salberger's global determinant method with the affine method of Heath-Brown for surfaces in \cite{B2}, in application to power-free values of polynomials in one or two variables. \\

We now give a complete statement of the global determinant method in the affine setting, recovering the generality given by Salberger in \cite{S2} and Xiao in \cite{X1} (in the setting of weighted projective spaces). \\ 

For positive numbers $B_1, \cdots, B_n$ all exceeding one, put
\[S(F; B_1, \cdots, B_n) = \{\Bx \in \bZ^n : |x_i| \leq B_i, F(\Bx) = 0\}.\]
Put $X$ for the hypersurface in $\bA^n$ defined by $F$. For a prime $p$ and a point $P \in X(\bF_p)$, put
\[S_p(F; B_1, \cdots, B_n; P) = \{\Bx \in S(F; B_1, \cdots, B_n) : \Bx \equiv P \pmod{p}\}.\]
For a vector of non-negative integers $\mathbf{e} = (e_1, \cdots, e_n)$, put $\Bx^{\mathbf{e}} = x_1^{e_1} \cdots x_n^{e_n}$. Let $\E$ be a finite set of vectors in $\bZ^n$ with non-negative entries. Put $s_P = \# S_p(F; \BB; P)$ and $E = \# \E$. Consider the $s_P \times E$ matrix $\M$ whose $ij$-th entry is the $i$-th monomial in $\E$ evaluated at the $j$-th element of $S_p(F; \Bx; P)$. \\ \\
We pick the same exponent set $\E$ as Heath-Brown in \cite{HB2}. In particular, write
\[F(x_1, \cdots, x_n) = \sum_{\mathbf{f}} a_{\mathbf{f}} \Bx^{\mathbf{f}},\]
and let $\P(F)$ be the Newton polyhedron of $F$. Put
\[T = \max_{\Bf \in \P(F)} \BB^{\Bf},\]
where $\BB = (B_1, \cdots, B_n)$ is a vector in $\bR^n$ with $B_i \geq 1$ for $1 \leq i \leq n$. Then there is at least one element $\Bm \in \P(F)$ such that $T = \BB^\Bm$. We pick a parameter $\lambda > 0$ and put $\Y = \lambda \log T$. We then define our exponent set $\E$ to be:
\begin{equation} \label{exponent set} \E = \left\{\Be \in \bZ^n : e_i \geq 0 \text{ for } 1 \leq i \leq n, \sum_{i=1}^n e_i \log B_i \leq \Y, e_i < m_i \text{ for at least one } i \right \}
\end{equation}
We shall abuse notation and also refer to $\E$ as a set of monomials. It was shown by Heath-Brown in \cite{HB2} that all non-trivial linear combinations of monomials whose exponent vectors lie in $\E$ leads to a polynomial which is not divisible by $F$. \\

Put
\begin{equation} \label{W} W = \exp \left( \left( \frac{\prod_{1 \leq i \leq n} \log B_i }{\log T} \right)^{1/(n-1)} \right).\end{equation}

We can now state the main result of this section, which is the global determinant method analogue of Heath-Brown's Theorem 15 in \cite{HB2}: 
\begin{theorem} \label{affine global} Let $\BB = (B_1, \cdots, B_n) \in \bR^{n}$ be a vector of positive numbers of size at least $1$. Let $X$ be a hypersurface defined over $\bZ$ in $\bA^n$ which is irreducible over $\bQ$ and defined by an irreducible polynomial $F$. Let $\UUU$ be a finite set of primes and put
\[\Q = \prod_{p \in \UUU}p. \]
For each prime $p \in \UUU$, let $P_p$ be a non-singular point in $X(\bF_p)$ and put
\[\UU = \{P_p : p \in \UUU\}.\]
Let $\E$ be as given in (\ref{exponent set}). Then
\begin{itemize}
\item[(a)] Let $\ep > 0$. If 
\[WT^\ep < \Q \leq WT^{2\ep}, \]
then there is a hypersurface $Y(\UU)$ containing $S(F; \BB; \UU)$, not containing $X$ and defined by a primitive polynomial $G$, whose degree is $O_{d,n,\ep}(1)$ and whose height is at most $O_{d,n,\ep}(\log T)$. 
\item[(b)] If $X$ is geometrically integral, then there exists a hypersurface $Y(\UU)$ containing $S(F; \BB, \UU)$, not containing $X$ and defined by a primitive polynomial $G$ whose degree satisfies
\[\deg G = O_{d,n} \left( (1 + \Q^{-1} W) \log T\Q\right).\]
\end{itemize}
\end{theorem}

Theorem \ref{affine global} can be proved using the same arguments as in \cite{S2} or \cite{X1}. The key insight is that the proofs in \cite{S2} and \cite{X1} producing large divisors of the discriminants essentially reduces to the affine case, so they remain valid if one starts with the affine case. \\ 

We will now need the following lemma which is a consequence of Theorem \ref{affine global}, and was stated without proof in \cite{B2}: 

\begin{lemma} Let $X$ be a geometrically integral affine surface of degree $d$ in $\bA^3$ defined by an integral polynomial $F$, and let $\BB = (B_1, B_2, B_3)$ be a vector of positive numbers exceeding one. Then for any $\ep > 0$ there exists a collection of integral polynomials $\Gamma$, defined over $\bZ$, such that
\begin{itemize}
\item[(a)] $\# \Gamma = O_{d,\ep} \left(W^{1 + \ep}\right),$
\item[(b)] Each polynomial $G \in \Gamma$ is co-prime with $F$, 
\item[(c)] The number of points in $X(\bZ; \BB)$ not lying on $\{G = 0 \}, G \in \Gamma$ is at most $O_{d,\ep} \left(W^{2 + \ep}\right)$, and
\item[(d)] Each polynomial $G \in \Gamma$ has degree $O_{d,\ep}(1)$.
\end{itemize}
\end{lemma}

\begin{proof} The proof given here follows the proof of Lemma 2.8 in \cite{S2} and Theorem 1.1 in \cite{X1}. By Theorem \ref{affine global}, there exists an affine surface $Y_0$, defined over $\bZ$, which contains $X(\bZ; \BB)$ but does not contain $X$ as a component, which is defined by a polynomial $G_0$ of degree $O_{d} \left(W^{1 + \ep}\right)$. \\ \\
Let $p_1 < \cdots < p_{t+1}$ be the sequence of increasing consecutive primes satisfying $p_1 > \log (B_1 B_2 B_3)$, and 
\[p_1 \cdots p_{t} \leq WT^\ep < p_1 \cdots p_{t+1}.\]
Write 
\[Q_t = \prod_{i =1 }^{t+1} p_i.\]
Then following the arguments in \cite{X1}, we see that $Q_t = O(W^{1 + \ep} \log W)$. \\ \\
We now begin constructing the set of polynomials $\Gamma$. By Theorem \ref{affine global}, for each point $P_{p_1} \in X(\bF_{p_1})$ there exists a surface $Y(P_{p_1})$ containing $X(\bZ;\BB; P_{p_1})$ but does not contain $X$ as a component of degree
\[O_{d} \left((1 + p_1^{-1} W) \log W\right).\]
$Y(P_{p_1})$ is defined by a polynomial $G_{P_{p_1}}$ with integer coefficients. Put $\Gamma(P_{p_1})$ for the set of irreducible factors $g$ of $G_{P_{p_1}}$ for which the intersection of $\{g = 0 \} \cap X$ contains a curve which lies in the intersection $Y_0 \cap X$. Now put $\Gamma^{(1)} = \bigcup_{P_{p_1} \in X(\bF_{p_1})} \Gamma(P_{p_1})$. Note that $\# \Gamma^{(1)}$ is bounded by the number of components of $Y_0 \cap X$, which is bounded by $O_d\left(W^{1 + \ep}\right)$ by B\'ezout's theorem. Moreover, since for each $g \in \Gamma^{(1)}$ is a divisor of $G_{P_{p_1}}$ for some $P_{p_1} \in X(\bF_{p_1})$, its degree is at most $O_d\left((1 + p_1^{-1} W) \log W\right)$. \\ \\
Likewise, for any point $P_1 \in X(\bF_{p_1})$ and $P_2 \in X(\bF_{p_2})$, there exists a polynomial $G_{P_1, P_2}$ defining a surface $Y(P_1, P_2)$ of degree $O_{d,\ep}\left((1 + (p_1 p_2)^{-1} W) \log W\right)$ which contains $X(\bZ; \BB; P_1, P_2)$. Now put $\Gamma^{(2)}$ for the collection of polynomials $g$ dividing $G_{P_1, P_2}$ for some $P_1 \in X(\bF_{p_1}), P_2 \in X(\bF_{p_2})$ and such that $\{g = 0\} \cap X$ contains a curve which is also contained in $Y_0 \cap X$. Again, we see that 
\[\deg g = O_{d,\ep} \left((1 + (p_1 p_2)^{-1} W) \log W\right)\]
and
\[\#\Gamma^{(2)} = O_{d,\ep}\left(W^{1 + \ep}\right).\]
We continue this process until we reach $t+1$, and set $\Gamma = \Gamma^{(t+1)}$. We then see that
\[\# \Gamma = O_{d,\ep} \left( W^{1 + \ep} \right)\]
and for each $g \in \Gamma$ we have
\[\deg g = O_{d,\ep}(1)\]
by part a) of Theorem \ref{affine global}. \\ \\
The points which lie on the complement $\Z$ of the union of the integral points on the surfaces defined by polynomials in $\Gamma$ in $X(\bZ; \BB)$ can be estimated as follows. If $\Bx \in \Z$, then there exists $0 \leq j \leq t+1$ such that the irreducible component of $Y(P_1, \cdots, P_j)$ containing $\Bx$ and the irreducible component of $Y(P_1, \cdots, P_{j+1})$ containing $\Bx$ differ. Then $\Bx$ lies on a 0-dimensional variety defined by $F, g_1, g_2$ where $g_1, g_2$ are distinct polynomials such that $g_1 | G(P_1, \cdots, P_j)$ and $g_2 | G(P_1, \cdots, P_{j+1})$. Thus, the number of such $\Bx$ is bounded by 
\[O_d \left((1 + (p_1 \cdots p_j)^{-1} W)\log W)((1 + (p_1 \cdots p_{j+1})^{-1} W)\log W \right)\]
by B\'ezout's theorem. Since $t = O_d\left(\frac{\log T}{\log \log T}\right)$, it follows that 
\[\# \Z = O_{d,\ep} \left(W^{2 + \ep}\right)\]
as desired. 
\end{proof}

The following lemma, given as Lemma 2 in \cite{B2}, then follows:

\begin{lemma} \label{browning} Let $f(x)$ be a polynomial of degree $d$ with integer coefficients, and let $B_1, B_2, B_3$ be positive real numbers. Put $M(f; \BB)$ for the number of integer solutions to the equation
\[f(x) = yz^k\]
satisfying 
\[B_1/2 \leq x < B_1, B_2/2 \leq y < B_2, B_3/2 \leq z < B_3.\]
Then 
\[M(f; \BB) = O_{d, \ep, k}\left( (B_1 B_2 B_3)^\ep W \left(W + B_1 B_3^{-k} + B_3^{1/d} \right) \right).\] 
\end{lemma} 

We remark that in \cite{B2} Browning had required that $f$ be irreducible, but this is not a necessary assumption. Indeed, the surface defined by $f(x) = yz^k$ is geometrically integral whenever $f$ is not identically zero, therefore the machinery established by Salberger in \cite{S2} will be applicable. \\

We may now finalize the proofs of Theorem \ref{MT1} and Theorem \ref{MT2} when $d \geq 5$. 

\subsection{Estimate of $N_3(B)$}\label{subsect N3} Following the strategy of Browning in \cite{B2}, we may fix $n-1$ variables, say $x_2, \cdots, x_n$, and reduce the problem to the single variable case. For a positive number $H$, put
\[R(F; B, H) = \#\{\Bx \in \bZ^n : |x_i| \leq B, y \ll B^d/H^k, H/2 \leq z < H, F(\Bx) = yz^k\}.\] 
We then see, by dyadic summation, that
\[N_3(B) \ll (\log B) \sup_{\xi_2 < H \ll B^{d/k}} R(F; B, H).\]
Put $F_{\Bm}(x) = F(x, m_2, \cdots, m_n)$, where $\Bm = (m_2, \cdots, m_n)$. Then
\begin{equation} R(F;B,H) \leq \sum_{|m_2|, \cdots, |m_n| \leq B} R(F_{\Bm}; B, H).\end{equation}
Put $H = B^\eta$ and $e = \deg F_{\Bm}(x)$. We can assume that the $x_1^d$ coefficient in $F$ is non-zero, eventually after applying unimodular transformation to the variables. Then we see from (\ref{W}) and our hypothesis that 
\[W =  \exp\left(\sqrt{\frac{(\log B)\eta (\log B) (d \log B - \eta k \log B)}{d \log B}} \right) = B^{\sqrt{\eta(1 - k\eta/d)}}.\]
We need to ensure that $F(x_1, \mathbf{n})$ is non-singular as a polynomial in a single variable. Let $\Delta(x_2, \cdots, x_n)$ be the discriminant of $F(x_1, \cdots, x_n)$ where we view the $x_2, \cdots, x_n$ as coefficients of $F(x_1, \mathbf{x})$. Then $F(x_1, \Bm)$ is singular if and only if $\Delta(m_2, \cdots, m_n) = 0$. Thus we may apply Ekedahl's sieve to conclude that such tuples are negligible. \\

We then apply Lemma \ref{browning} to see that
\[R(F_{\Bm}; B, H) \ll \left(BH \frac{B^d}{H^k}\right)^\ep B^{\sqrt{\eta (1 - k\eta/d)}}\left(B^{\sqrt{\eta(1 - k \eta/d)}} + B^{1 - k \eta} + B^{\eta/d}\right). \]
We note that for any $v > 0$, the quadratic function
\[y = x(1 - vx)\]
is concave down and decreasing on the positive real line. Therefore, the maximum value for the term $\sqrt{\eta(1 - k\eta/d)}$ is achieved when $\eta$ is minimum. Since $H \gg \xi_2$ it follows that $\eta \geq 1$, so the maximum occurs when $\eta = 1$. Thus the maximum value is
\[\sqrt{1 - k/d}.\]
Note that 
\[\sqrt{1 - \frac{k}{d} }> 1 - k \text{ and } \sqrt{1 - \frac{k}{d} } > \frac{1}{d}.\]
It thus follows that
\[N_3(B) = O_{d,n,\ep} \left(B^{n-1 + 2 \sqrt{\eta(1 - k\eta/d)} + \ep}\right).\]
We want to have $N_3(B) = O(B^{n-\epsilon})$, so that to prove both Theorem \ref{MT1} and Theorem \ref{MT2}, we shall simply insist that
\[2 \sqrt{1 - k/d} < 1,\]
which is equivalent to $k > 3d/4$. This finishes the proofs.

\subsection*{Acknowledgements} We would like to thank the referee for her/his very careful reading and pointing out several issues which were addressed in the current version. K. Lapkova is supported by a Hertha Firnberg grant [T846-N35] of the  Austrian Science Fund (FWF).


\begin{thebibliography}{10}

\bibitem{Bha}
M.~Bhargava, \emph{The geometric sieve and the density of squarefree values of invariant polynomials}, 	arXiv:1402.0031 [math.NT]. 

\bibitem{BhaSha}
M.~Bhargava, A.~Shankar, \emph{Binary quartic forms having bounded invariants, and the boundedness of the average rank of elliptic curves}, Annals of Mathematics \textbf{181} (2015), 191-242. 

\bibitem{BSW}
M.~Bhargava, A.~Shankar, X.~Wang, \emph{Squarefree values of polynomial discriminants I}, 	arXiv:1611.09806 [math.NT]. 

\bibitem{B2}
T.~D.~Browning, \emph{Power-free values of polynomials}, Arch.~Math. (2) \textbf{96} (2011), 139-150. 

\bibitem{Erd}
P.~Erd\H{o}s, \emph{Arithmetical properties of polynomials}, J. London Math. Soc. \textbf{28} (1953), 416-425.

\bibitem{Fil}
M.~Filaseta, \emph{Powerfree values of binary forms}, Journal of Number Theory \textbf{49} (1994), 250-268.

\bibitem{Gre}
G.~Greaves, \emph{Power-free values of binary forms}, Q.~J.~Math, (2) \textbf{43} (1992), 45-65.

\bibitem{HB1}
D.~R.~Heath-Brown, \emph{The density of rational points on curves and surfaces}, Annals of Mathematics \textbf{155} (2002), 553-598.

\bibitem{HB2}
D.~R.~Heath-Brown, \emph{Counting rational points on algebraic varieties}, Analytic number theory, 51-95, Lecture Notes in Math., 1891, Springer, Berlin, 2006. 

\bibitem{Hoo1}
C.~Hooley, \emph{On the power free values of polynomials}, Mathematika \textbf{14} (1967), 21-26.

\bibitem{hooley76} C. Hooley, Applications of sieve methods to the theory of numbers, \emph{Cambridge Tracts in Mathematics 70}, Cambridge Univ.~Press, 1976

\bibitem{Hoo2}
C.~Hooley, \emph{On the power-free values of polynomials in two variables}, Analytic number theory, 235-266, Camb.~Univ.~Press, 2009.

\bibitem{Hoo3}
C.~Hooley, \emph{On the power-free values of polynomials in two variables: II}, Journal of Number Theory, \textbf{129} (2009), 1443-1455.

\bibitem{leboudec} P.~Le Boudec, \emph{Power-free values of the polynomial $t_1\ldots t_r -1$},Bull.~Aust.~Math.~Soc., \textbf{85} (2012),154-163

\bibitem{lapkova} K.~Lapkova, \emph{On the k-free values of the polynomial $xy^k+C$}, Acta Math.~Hungar., \textbf{149} (1) (2016), 190-207


\bibitem{Po} B.~Poonen, \emph{Squarefree values of multivariable polynomials}, Duke Math.~J.,  \textbf{118} (2003), 353-373

\bibitem{S2}
P.~Salberger, \emph{Counting rational points on projective varieties}, Preprint, 2009.

\bibitem{X1}
S.~Y.~Xiao, \emph{Power-free values of binary forms and the global determinant method}, Int.~Math.~Res.~Notices, Issue 16 Volume 2017, 5078-5135. 

\bibitem{X2}
S.~Y.~Xiao, \emph{Square-free values of decomposable forms}, Canadian Journal of Mathematics, 70(6), 1390-1415, DOI:  10.4153/CJM-2017-060-4. 

\end{thebibliography}
\end{document}